\documentclass[11pt]{amsart}
\usepackage[utf8]{inputenc}

\title[The group law in the Weyl group]{Geometric characterization of the group law\\ in the Weyl group}

\author{Kenta Suzuki}

\calclayout

\usepackage{amsmath, amssymb, amsthm,tikz-cd,mathrsfs,hyperref,colonequals}
\usepackage[shortlabels]{enumitem}

\numberwithin{equation}{section}

\newtheorem{thm}{Theorem}[section]

\newtheorem{lemma}[thm]{Lemma}
\newtheorem{cor}[thm]{Corollary}

\theoremstyle{definition}

\newtheorem{defn}{Definition}
\newtheorem{example}[thm]{Example}

\begin{document}

\begin{abstract}
    Let $G$ be a reductive group with Borel $B$ and Weyl group $W$. Then $B$-double cosets in $G$ are indexed by the Weyl group, say $O(w)$ for $w\in W$. Then we prove the minimal $B$-double coset in the convolution $O(w_1)*O(w_2)$ is $O(w_1w_2)$, which gives a geometric characterization of multiplication in $W$. This defines the abstract Weyl group $\mathbf W$ which is a Coxeter group acting on the abstract Cartan $\mathbf T$.
\end{abstract}

\maketitle

\section{Introduction}

Let $G$ be a reductive group. Traditionally, the Weyl group is defined as, given a choice of maximal torus $T$ of $G$,
\[
W(G,T):=N_G(T)/T.
\]
However, the definition is not canonical: any two choices $T_1$ and $T_2$ of maximal tori are conjugate by some $g\in G$, say $\mathrm{ad}(g)T_1=T_2$, and $\mathrm{ad}(g)$ provides an isomorphism $W(G,T_1)\simeq W(G,T_2)$. However, $g$ is only unique up to left multiplication by $N_G(T_2)$, so this only defines the Weyl group of $G$ up to conjugation. 

Just as we can define the \emph{abstract Cartan} $\mathbf T$ of $G$ as $\mathbf T_B:=B/[B,B]$ for any choice of Borel subgroup $B$ in $G$, so that for any two Borels $B_1$ and $B_2$ of $G$ there is a \emph{canonical} isomorphism $\mathbf T_{B_1}\simeq\mathbf T_{B_2}$, we hope to define the \emph{abstract Weyl group} $\mathbf W$. Let $\mathcal B$ be the flag variety, parameterizing Borel subgroups of $G$. Upon a choice of base point $B\in\mathcal B$, there is an isomorphism $\mathcal B\simeq G/B$. Then $G$-orbits in $\mathcal B\times\mathcal B$ are indexed by a finite set $\mathbf W$:
\[
\mathcal B\times\mathcal B=\bigsqcup_{w\in \mathbf W}O(w).
\]
Upon a choice of pinning $T\subset B\subset G$, there is a bijection
\[
G\backslash(\mathcal B\times\mathcal B)\simeq B\backslash G/B\simeq W(G,T),
\]
which provides a bijection between the abstract Weyl group $\mathbf W$ and $W(G,T)$.

Many properties of the Weyl group can be recovered from this geometric characterization:
\begin{itemize}
    \item for $w\in \mathbf W$, the length $\ell(w)$ is $\dim O(w)-\dim\mathcal B$;
    \item the simple reflections $\mathbf S\subset\mathbf W$ are those $w\in\mathbf W$ such that $\ell(w)=1$;
    \item the given $w_1,w_2\in\mathbf W$, the Bruhat order is defined as $w_1<w_2$ whenever $O(w_1)\subset\overline{O(w_2)}$; and
    \item for $w_1,w_2\in\mathbf W$, the Demazure product $w_1\star w_2$ is characterized by $O(w_1\star w_2)$ being the unique open $G$-orbit in the convolution
    \[
    O(w_1)*O(w_2):=\{(B_1,B_2)\in\mathcal B\times\mathcal B:\exists B\in\mathcal B\text{ such that $(B_1,B)\in O(w_1)$ and $(B,B_2)\in O(w_2)$}\}.
    \]
    Alternatively, $O(w_1)*O(w_2)=q(p^{-1}(O(w_1)\times O(w_2)))$ is given by the correspondence
    \[
    \mathcal B\times\mathcal B\times\mathcal B\times\mathcal B\xleftarrow{p}\mathcal B\times\mathcal B\times\mathcal B\xrightarrow{q}\mathcal B\times\mathcal B,
    \]
    where $p(B_1,B,B_2)=(B_1,B,B,B_2)$ and $q(B_1,B,B_2)=(B_1,B_2)$.
\end{itemize}
We prove the following characterization of the group law in $\mathbf W$:
\begin{thm}\label{main-thm0}
    The convolution $O(w_1)*O(w_2)$ has a unique closed $G$-orbit $O(w_1w_2)$. This gives $\mathbf W$ a group structure. Under a choice of pinning $T\subset B\subset G$ the identification $(\mathbf W,\mathbf S)\simeq(W(G,T),S(G,T))$ is an isomorphism of Coxeter groups.
\end{thm}
Moreover, just as $W(G,T)$ acts on $T$, the abstract Weyl group $\mathbf W$ acts on the abstract Cartan $\mathbf T$:
\begin{defn}
    For any $w\in\mathbf W$, let $(B_1,B_2)\in O(w)$. Then $B_1\cap B_2\to \mathbf T_{B_1}:=B_1/[B_1,B_1]$ is surjective, so there is an isomorphism $B_1\cap B_2/U_{B_1\cap B_2}\simeq\mathbf T_{B_1}$, where $U_{B_1\cap B_2}$ is the unipotent radical of $B_1\cap B_2$. Let $w$ act on $\mathbf T$ by the composition
    \[
    \mathbf T_{B_1}\xleftarrow{\sim}B_1\cap B_2/U_{B_1\cap B_2}\xrightarrow\sim\mathbf T_{B_2}.
    \]
\end{defn}
Then:
\begin{thm}
    Let $T\subset B\subset G$ be a pinning. Then the isomorphism $T\simeq\mathbf T$ is equivariant under $W(G,T)\simeq\mathbf W$.
\end{thm}

\subsection{Acknowledgment} The author thanks Zhiwei Yun for asking this question in his class 18.758 at MIT in Spring 2025, and George Lusztig for helpful discussions.

\section{A combinatorial characterization of the convolution}
Recall that $W=W(G,T)$ with the simple reflections $S=S(G,T)$ has the structure of a Coxeter group. Now given a $w\in W$ and $s\in S$ such that $ws>w$, we have
\[
O(ws)=O(w)*O(s).
\]
Thus, when $s_1\cdots s_n$ is a reeduced expression, we can write
\[
O(x)*O(s_1\cdots s_n)=O(x)*O(s_1)*\cdots*O(s_n).
\]
Moreover, 
\[
O(x)*O(s)=\begin{cases}
    O(xs)&\text{if }xs>x\\
    O(xs)\cup O(x)&\text{if }xs<x.
\end{cases}
\]
Thus, we can define:

\begin{defn}
    Let $(W,S)$ be a Coxeter group. For $x_1,x_2\in W$, define the set $x_1*x_2$ inductively on $\ell(x_2)$ as follows. Let $x_1*1=\{x_1\}$, and when $x_2\ne 1$ let $x_2=x_2's$ for $s\in S$ where $\ell(x_2')<\ell(x_2)$. Then
\[
x_1*x_2:=(x_1*x_2')s\cup\{w\in x_1*x_2':ws<w\}.
\]
\end{defn}
Then the above discussion shows:
\begin{lemma}\label{lem:translation}
    Let $G$ be a reductive group with a maximal torus $T$. Let $W=W(G,T)$ and $S=S(G,T)$, so $(W,S)$ is a Coxeter group. Then for $w_1,w_2\in W$,
    \[
    O(w_1)*O(w_2)=\bigsqcup_{x\in w_1*w_2}O(x).
    \]
\end{lemma}
By Lemma~\ref{lem:translation}, Theorem~\ref{main-thm0} reduces to the following combinatorial statement about Coxeter groups:
\begin{thm}\label{main-thm}
    Let $(W,S)$ be a Coxeter group. Then for all $x_1,x_2\in W$, the set $x_1*x_2$ has a unique minimal element with respect to the Bruhat order, $x_1x_2$.
\end{thm}

\section{The proof of Theorem~\ref{main-thm}}

The Bruhat order is defined combinatorially as follows:
\begin{defn}[{\cite[\S5.9]{humphreys}}]
    Let $T$ be the set of reflections in $W$ with respect to roots. Let $w'\to w$ if $w=w't$ and $\ell(w')<\ell(w)$ for some $t\in T$. Let $w'<w$ if there is a sequence $w'=w_0\to w_1\to\cdots\to w_m=w$.
\end{defn}

Also recall the following lemma:
\begin{lemma}[{\cite[Proposition~5.9]{humphreys}}]\label{lem1}
    Let $w'\le w$ and $s\in S$. Then either $w's\le w$ or $w's\le ws$.
\end{lemma}
The following is the key consequence:
\begin{lemma}\label{lem3}
    Suppose $w,w'\in W$ are such that $w'\to w$ and there exists a $s\in S$ such that $w's\ne w$. Then $w's\to ws$.
\end{lemma}
\begin{proof}
    By Lemma~\ref{lem1} we know $w's<ws$ or $w's\le w$. In the former case we are done so assume $w's\le w$. By the definition of the Bruhat order there exists a $t\in T$ such that $w=w't$. Since $ws=w's\cdot(sts)$ and $sts\in T$, to prove $w's\to ws$ it suffices to check $\ell(w's)<\ell(ws)$.

    Since $w's\le w$ and $w's\ne w$, we know $\ell(w's)\le \ell(w)-2$, so
    \[
    \ell(ws)\ge\ell(w)-1>\ell(w's),
    \]
    as desired. 
\end{proof}
As a corollary,
\begin{cor}\label{cor1}
    Suppose $u,x\in W$ and $s\in S$ are such that $us<u$ and $sx>x$. Then
    \[
    usx\to ux.
    \]
\end{cor}
\begin{proof}
    We prove this by induction on $\ell(x)$. Let $x=x's'$ where $\ell(x')<\ell(x)$. Then by our inductive hypothesis $usx'\to ux'$. Now $sx\ne x'$ since $\ell(sx)>\ell(x)$ while $\ell(x')<\ell(x)$. Thus 
    \[
    usx's'=usx\ne ux',
    \]
    so by Lemma~\ref{lem3} we have $usx's'=usx\to ux's'=ux$, as desired.
\end{proof}

Now, we can prove Theorem~\ref{main-thm}:
\begin{proof}[Proof of Theorem~\ref{main-thm}]
    Let $x_2=s_1\cdots s_n$ be a reduced expression for $x_2$. 
    We prove by induction on $j$ that the Bruhat minimal element of $x_1*x_2$ lies in the subset
    \[
    (x_1*s_1\cdots s_{n-j})s_{n-j+1}\cdots s_n=\{us_{n-j+1}\cdots s_n:u\in x_1*s_1\cdots s_{n-j}\}.
    \]
    Indeed for $u\in x_1*s_1\cdots s_{n-j-1}$ if $us_{n-j}>u$ then $u*s_{n-j}=\{us_{n-j}\}$ so
    \[
    us_{n-j}s_{n-j+1}\cdots s_n\in (x_1*s_1\cdots s_{n-j-1})s_{n-j}\cdots s_n.
    \]
    On the other hand if $us_{n-j}<u$ then $u*s_{n-j}=\{u,us_{n-j}\}$. Then by Corollary~\ref{cor1}, note that
    \[
    us_{n-j}\cdots s_n<us_{n-j+1}\cdots s_n,
    \]
    Bruhat minimal elements must again live in
    \[
    (x_1*s_1\cdots s_{n-j-1})s_{n-j}\cdots s_n.
    \]
    When $j=n$ this shows the unique Bruhat minimal element of $x_1*x_2$ is $x_1x_2$.
\end{proof}

\section{Further questions}
A consequence of Theorem~\ref{main-thm0} is:
\begin{cor}
    For any $w_1,w_2\in\mathbf W$, the $x\in W$ such that the $G$-orbit $O(x)$ lies in the convolution $O(w_1)*O(w_2)$ are in the Bruhat interval $[w_1w_2,w_1\star w_2]$.
\end{cor}
It would be interesting to give a complete characterization of the $G$-orbits in $O(w_1)*O(w_2)$. The Bruhat interval is not always exhausted:

\begin{example}
    When $G=\mathrm{GL}_3$ so $W=\langle s_1,s_2|s_1s_2s_1=s_2s_1s_2\rangle$. Then
    \[
    s_1s_2*s_2s_1=s_1*s_1\cup s_1*s_2s_1=\{1,s_1,s_1s_2s_1\},
    \]
    which does not contain $s_2\in[s_1s_2s_2s_1,s_1s_2\star s_2s_1]=[1,s_1s_2s_1]$.
\end{example}

\bibliographystyle{amsalpha}
\bibliography{bibfile}
\end{document}